\documentclass[aop,preprint]{imsart}

\arxiv{1501.04595}

\RequirePackage[OT1]{fontenc}
\RequirePackage{amsthm,amsmath,amssymb}
\RequirePackage[colorlinks,citecolor=blue,urlcolor=blue]{hyperref}
\usepackage{stmaryrd}
\usepackage[sans]{dsfont}
\usepackage[english]{babel}
\usepackage{latexsym}
\usepackage{bbm}
\usepackage[mathcal]{euscript}
\usepackage{graphicx}
\usepackage{pstricks}
\usepackage{pstricks-add}
\usepackage{pst-plot}
\usepackage{enumerate}
\usepackage{multicol}
\usepackage{subfigure}
\newcommand{\e}{\varepsilon}   

\newcommand{\RR}{\mathbb{R}}    
\newcommand{\ind}{\mathbbm{1}} 
\newcommand{\C}{\mathcal{C}}    

\newcommand{\M}{\mathcal{M}}    %
\newcommand{\NN}{\mathbb{N}}

\renewcommand{\O}{\Omega}    %

\renewcommand{\SS}{\mathbb{S}}
\newcommand{\PP}{\mathbb{P}}    
\newcommand{\EE}{\mathbb{E}}     

\newcommand{\sub}{\subseteq}    

\providecommand{\rpar}[1]{\left( #1 \right)}               
\providecommand{\kpar}[1]{\left\{ #1 \right\}}              
\providecommand{\spar}[1]{\left[  #1 \right]}              
\providecommand{\abs}[1]{\left\vert #1 \right\vert}           

\providecommand{\dd}[2] {\frac{d #1 }{d #2}}            

\numberwithin{equation}{section}
\theoremstyle{plain}
\newtheorem{theorem}{Theorem}[section]
\newtheorem{lemma}[theorem]{Lemma}

\newtheorem{corollary}[theorem]{Corollary}
\newtheorem{remark}{Remark}[section]
\theoremstyle{definition}

\newcommand{\dct}{Dominated Convergence Theorem}

\begin{document}

\begin{frontmatter}

\title{Asymptotics for the heat kernel in multicone domains}
\runtitle{Heat kernel in multicone domains}

\begin{aug}
\author{\fnms{Pierre} \snm{Collet}\thanksref{t1}\ead[label=e1]{pierre.collet@cpht.polytechnique.fr}},
\author{\fnms{Mauricio} \snm{Duarte}\thanksref{t2}\ead[label=e2]{mauricio.duarte@unab.cl}},
\author{\fnms{Servet} \snm{Mart\'inez}\thanksref{t1}\ead[label=e3]{smartine@dim.uchile.cl}}, 
\author{\fnms{Arturo} \snm{Prat-Waldron}\ead[label=e4]{arturo@mpim-bonn.mpg.de}}
\and
\author{\fnms{Jaime} \snm{San Mart\'in}\thanksref{t1}\ead[label=e5]{jsanmart@dim.uchile.cl}}

\thankstext{t1}{We thank the Center for Mathematical Modeling (CMM) Basal CONICYT Program PFB 03.}
\thankstext{t2}{Thanks for the support from proyect FONDECYT 3130724, and the Programa Iniciativa Cientifica Milenio grant number NC130062 through the Nucleus Millenium Stochastic Models of Complex and Disordered Systems..}

\runauthor{P. Collet et al.}

\affiliation{Ecole Polytechnique, Universidad Andres Bello, Centro de Modelamiento Matem\'atico and Max Planck Institute for Mathematics}

\address{P. Collet,\\CNRS Physique Th\'eorique\\
Ecole Polytechnique,\\
91128 Palaiseau cedex, France,\\ \printead{e1}}

\address{M. Duarte,\\
Departamento de Matem\'atica,\phantom{ali\hspace{0.665cm}gn}\\
Universidad Andres Bello,\\
Rep\'ublica 220, Santiago, Chile,\\ \printead{e2}}

\address{S. Mart\'inez, J. San Mart\'in,\\
Departamento de Ingenier\'ia Matem\'atica,\\
Facultad de Ciencias F\'isicas y Matem\'aticas,\\ 
Universidad de Chile, \\
Beauchef 851, torre norte, piso 5,\\
Santiago, Chile\\ \printead{e3}\\
\printead{e5}}

\address{A. Prat-Waldron,\\
Max Planck Institute for Mathematics,\\ Vivatsgasse 7,\\
53111 Bonn, Germany,\\ \printead{e4}}

\end{aug}

\begin{abstract}
. A multi cone domain $\Omega \subseteq \mathbb{R}^n$ is an open, connected set that resembles a finite collection of cones far away from the origin. We study the rate of decay in time of the heat kernel $p(t,x,y)$ of a Brownian motion killed upon exiting $\Omega$, using both probabilistic and analytical techniques. We find that the decay is polynomial and we characterize  $\lim_{t\to\infty} t^{1+\alpha}p(t,x,y)$ in terms of the Martin boundary of $\Omega$ at infinity, where $\alpha>0$ depends on the geometry of  $\Omega$. We next derive an analogous result for $t^{\kappa/2}\mathbb{P}_x(T >t)$, with $\kappa = 1+\alpha - n/2$, where $T$ is the exit time form $\Omega$. Lastly, we deduce the renormalized Yaglom limit for the process conditioned on survival.
\end{abstract}



\end{frontmatter}

\section{Introduction}
\label{se:intro}
\newcommand{\bd}{\mathfrak{S}}
\renewcommand{\d}{\mathfrak{D}}
\renewcommand{\k}{N}
\newcommand{\zed}{2^\alpha \Gamma(1+\alpha)}
\setlength{\parskip}{0.2cm}

Let $O$ be a domain (open and connected set) in $\RR^n$, regular  for the Dirichlet problem. Consider an $n-$dimensional Brownian motion $B_t$ starting from the interior of  $O$,  with exit time $T^O$. 
The heat kernel $p^O(t,x,y)$ is the Radon-Nikodym derivative of the Borel measure $ A\mapsto\PP_x(B_t\in A, T^O>t)$ with respect to the $n-$dimensional Lebesgue measure, and it is characterised to be the fundamental solution of the heat equation with Dirichlet boundary condition, that is: as a function of $(t,y)$ it solves the heat equation $\partial_t u=\frac12\Delta u$, it vanishes continuously on $\partial O$, and it satisfies the initial condition $u(0,y)=\delta_x(y)$. 

It is well known that $p^O(t,x,y)$ tends to zero as time grows to infinity. A classical problem is to find the exact asymptotic (in time) for the decay of the heat kernel and the survival probability. This is well understood  for bounded domains (see \cite{Pin90} and \cite{Pin95}). For results in some planar domains we refer the reader to \cite{BaD89}. The large time asymptotic problem is treated in \cite{Pin92} for a large class of (non symmetric) diffusions under some integrability conditions on the ground state. Exact asymptotic are computed for Benedicks domains in \cite{CMS99}, and for exterior domains in \cite{CMS00}. Our work focuses on finding the exact asymptotic in time for $p^\O(t,x,y)$ and $\PP_x(T^\O>t)$ for a multicone domain $\O$, which we define next.

Let $\SS^{n-1}=\kpar{x\in\RR^n : \abs{x}=1}$ be the unit sphere in $\RR^n$. Points in $\RR^n$ will be regarded as $x=r\theta$, where $r=\abs{x}$ and $\theta\in\SS^{n-1}$. Given a Lipschitz, proper subdomain $\d$ of $\SS^{n-1}$, and a vector $a\in\RR^n$, a truncated cone with opening $\d$ and vertex $a$ is the set 
$$
C(a,\d,R)=\kpar{a+x : x=r\theta\in\RR^n : r>R,\ \theta\in \d},
$$ 
where $R\ge 0$.  When $R>0$, the set $\bd = a+R \d$ will be called the \emph{base} of the truncated cone. When $R=0$, we will refer to the set in the previous display as {cone with vertex $a$}.

In the same context as above, given a base $\bd=a+R\d$, let $0<\lambda^1<\lambda^2\leq \lambda^3\leq\cdots$ be the eigenvalues of the Laplace-Beltrami operator on $\d$, with  corresponding orthonormal basis $\{m^1,m^2,m^3,\ldots\}$ of $L^2(\d,\sigma)$, where $\sigma$ is the surface measure on $\SS^{n-1}$. Let $\alpha^i = \rpar{\lambda^i+(\frac{n}2-1)^2}^{1/2}$. We define the character of the base $\bd$ as the number $\alpha=\alpha(\d)=\alpha^1$. The character of the truncated cone $C(a,\d,R)$ is also defined as $\alpha$. 

A multicone domain $\O\sub\RR^n$ is a connected, open set such that there exists a bounded domain $\O_0\sub\O$ and finitely many truncated cones $\O_j = C(a_j,\d_j,R_j)$, with $j=1,\ldots \k$, such that $\O_j\cap \O_i=\emptyset$ for $1\leq j < i \leq \k$, and
$$
\O\setminus\overline\O_0 = \bigcup_{j=1}^\k \O_j. 
$$
Here $\overline{\O}_0$ is the closure of the set $\O_0$. The set $\O_0$ will be called the \emph{core}, and for $j\geq 1$, the sets $\O_j$ are called \emph{branches} of the multi-cone set. Notice that by construction, the branches are disjoint from the core. Also, we will denote the base of the truncated cone $\O_j$ by $\bd_j$. Without loss of generality, we can assume that $R_j=1$, which makes the exposition that follows much easier. 
%
%
%
The character of the truncated cone $\O_j$ will be denoted by $\alpha_j$. We define the character of the multicone $\O$ as the number $\alpha=\min\kpar{\alpha_j : j=1,\ldots,\k}$. An index $l$ such that $\alpha_l=\alpha$ will be called maximal. We denote by $\mathcal{M}$ the set of maximal indices.

To state the main results of this article, we need to introduce the Martin boundary at infinity for $\O$.




It is well known that there is a unique minimal harmonic function $w$ on a  cone with vertex $\C_0=C(a,\d,0)$ that vanishes continuously on $\partial \C_0$. Actually, there is only one positive harmonic function in $\C_0$ that vanishes continuously on its boundary (Theorem 1.1 in \cite{Anc12}). For $x=a+\abs{x-a}\theta\in \C_0$ this function is given by:
\begin{align}
\label{eq:coneharmonic}
v(x)=\abs{x-a}^{\alpha-\rpar{\frac{n}{2}-1}}m^1(\theta),
\end{align} 
where $\alpha$ is the character of $\d$ and $m^1$ is the first eigenfunction of the Laplace-Beltrami operator on $\d$. Notice how we have chosen to normalize $w$ in terms of the normalization of $m^1$ in $L^2(\d,\sigma)$. In order to simplify our exposition, we set $\kappa = 1+\alpha-n/2$, so that $v(x)=\abs{x-\alpha}^{\kappa}m^1(\theta)$. 

Similarly, if $\C=C(a,\d,R)$ is a truncated cone, there is a unique (minimal) positive harmonic function $w$ in $\C$ that vanishes continuously on $\partial\C$, which is defined as follows: let $T^\C$ be the exit time of a Brownian motion $B_t$ from the cone $\C$. Then
\begin{align}
\label{eq:truncatedharmonic}
w(x)=v(x) - \EE_x(v(B_{T^\C})),\qquad x\in\C.
\end{align}

Let $w_j$ be the unique minimal harmonic function in $\O_j$.
By a standard balayage argument \cite{Mat56}, one can extend $w_j$ to a minimal harmonic function in $\O$. Such extension is given by
\begin{align}
\label{eq:harmonic_extension}
u_j(x) = w_j(x)\ind_{\O_j}(x) + \frac12 \int_{\bd_j} G(x,y) \partial_n{w_j}(y)\sigma_j(dy),\qquad x\in\O,
\end{align}
where $\partial_n$ denotes the (inward) normal derivative on $\bd_j$, and $\sigma_j$ is the translation of $\sigma$ by $a_j$, and $G$ is the Green function of the domain $\O$:
$$
G(x,y) = \int_0^\infty p(t,x,y) dt.
$$

Reciprocally, we have that
\begin{align}
\label{eq:harmonic_extension_2}
w_j(x) &= u_j(x) -\EE_x u(B_{T^j}),\qquad x\in \O_j,
\end{align}
where $B$ is an $n-$dimensional Brownian motion, stopped at its exit time $T^j$ from $\O_j$. 

It is direct to verify from the last two equations that the function $u_j$ is bounded in $\O\setminus\O_j$, and satisfies that for $x=a_j+r\theta$, 
\begin{align}
\label{eq:limit_1}
\lim_{r\to\infty} \frac{u_j(a_j+r\theta)}{w_j(a_j+r\theta)} =1, 
\end{align}
for fixed $\theta\in \d_j$.


We are ready to state the main results of this paper.

\begin{theorem}
\label{th:mainkernel}
Let $\O$ be a multicone domain with branches $\O_1,\ldots,\O_\k$. Let $\alpha>0$ be the character of $\O$, and let $\M$ be the set of maximal indices. Then,
\begin{align}
\label{eq:mainkernel}
\lim_{t\to\infty} t^{1+\alpha}p(t,x,y) = \frac{1}{2^\alpha\Gamma(1+\alpha)} \sum_{l\in\M} u_l(x)u_l(y),
\end{align} 
The limit is in the topology of uniform convergence on compact sets.
\end{theorem}

\begin{theorem}
\label{th:mainexit}
Let $\O$ be a multicone domain with branches $\O_1,\ldots,\O_\k$. Let $\alpha>0$ be the character of $\O$, and let $\M$ be the set of maximal indices. Set $\kappa = 1+\alpha-n/2$. Then
\begin{align}
\label{eq:mainexit}
\lim_{t\to\infty} t^{\kappa/2}\PP_x(T>t) =\frac{\Gamma\rpar{\frac{\kappa+n}{2}}}{2^{\kappa/2}\Gamma\rpar{\kappa + \frac{n}{2}}} \sum_{l\in\M} \rpar{\int_{\d_l} m^1_l(\theta) \sigma(d\theta)} u_l(x).
\end{align}
The limit is in the topology of uniform convergence on compact sets.
\end{theorem}

\begin{theorem}
\label{th:mainyaglom}
Let $\O$ be a multicone domain with character $\alpha>0$, and set $\beta = 1+\alpha+n/2$. Fix $x\in\O$, and $1\leq j\leq \k$. For each $y=\abs{y}\theta$, with $\theta\in \d_j$, we have that  $a_j+\sqrt{t}y \in \O_j$, for large enough values of $t$, and
\begin{align}
\label{eq:mainyaglom}
\lim_{t\to \infty} t^{\beta/2} p(t,x,a_j + \sqrt{t}y) &= \ind_{\M}(j) \frac{u_j(x)v_j(y)}{\zed} e^{-\abs{y}^2/2}.
\end{align}
The limit is in the sense of uniform convergence on compact sets on the variables $x$ and $y$.
\end{theorem}

The paper is organized as follows. Section \ref{se:results} lists some key results that we take from the literature on heat kernels for killed diffusions, in particular, subsection \ref{ss:vertex} includes our main theorems for the case of a cone with vertex. Section \ref{se:truncated} deals with the asymptotics for truncated cones, and Section \ref{se:multicone} includes some lemmas leading up to the proofs of the main theorems, which are contained at the end of Section \ref{se:multicone} for the decay of the heat kernel, and in Section \ref{se:exit} for the decay of the survival probability. Finally, Section \ref{se:yaglom} includes the proof of Theorem \ref{th:mainyaglom} and discusses a renormalized Yaglom limit for the killed Brownian motion.

\section{Preliminary results}
\label{se:results}

In what follows we make the following simplifications, in order to keep the exposition clear. We set $T=T^\O$, $T^j=T^{\O_j}$ and denote by $p$ and $p^j$ the respective heat kernels. In some of the formulas below, integrals over $\bd_j$ are understood to be with respect to the translated measure $\sigma^j$, but we will omit the index since the dependence on $j$ is clear from the domain of integration. Also, we will abuse the notation by omitting the vector $a_j$ form all the formulas involving functions in cones, since its inclusion affects all such functions by a simple translation of coordinates. In particular, we will write $p^j(t,x,y)$ for $x=\abs{x}\theta, y=\abs{y}\eta$ for $\theta,\eta\in \d_j$ instead of $p^j(t,x+a_j,y+a_j)$ in order to simplify our exposition. In this spirit, we will often say that $x\to\infty$ radially in $\O_j$ to mean that $x=a_j+r\theta$, and $r\to\infty$.

We start by listing some general properties of heat kernels in unbounded domains.

\begin{lemma}[Lemma 2.1 in \cite{CMS00}]
\label{le:cms00}
Let $O$ be a regular domain for the Dirichlet problem. Let $u(t,x)$ be a positive solution of the heat equation in $\RR_+\times O$, and consider a function $a:\RR_+\to\RR_+$ such that
\begin{align}
\label{eq:cms00}
\sup_{t\ge t_0,\abs{s}\leq 2} \frac{ a(t+s) }{ a(t) } <\infty,
\end{align}
for some $t_0>0$. Further, assume that the family of functions $\kpar{a(t)u(t,\cdot) : t\ge t_0}$ is bounded on compact sets. Then, the family $\kpar{a(t)u(t,\cdot): t\ge t_0+1}$ is equicontinuous on compact sets of $O$.
\end{lemma}

The next lemma corresponds to Lemmas 2.1-2.4 in \cite{CMS99}, which are proved for Benedicks domains in $\RR^n$. Nonetheless, the proofs work in a much more general setting, as long as the domain $O$ is  a regular domain for the Dirichlet problem, with infinite interior radius.

\begin{lemma}[Lemmas 2.1-2.4 in \cite{CMS99}]
\label{le:cms99}
In the same setting of Lemma \ref{le:cms00}, for $x,y\in O$ and $s\in\RR$ we have
\begin{align}
\label{eq:cms99}
\lim_{t\to\infty} \frac{ p(t+s,x,y) }{ p(t,x,y) } =1.
\end{align}
The limit is uniform in compact sets of $\overline O$. Also, the map $t\mapsto p(t,x,x)$ is decreasing.
\end{lemma}

\begin{lemma}
\label{le:limitharmonic}
In the same setting as in Lemma \ref{le:cms99}, further assume that for all $s\in\RR$,
\begin{align}
\label{eq:assumtion_a}
\lim_{t\to\infty} \frac{ a(t+s) }{a(t)} =1.
\end{align}
If $a(t)p(t,x,y)\leq C_x^{1+\abs{y}}$ for large enough $t$, then any limit point of $a(t)p(t,\cdot,\cdot)$ (in the topology of uniform convergence on compact sets) has the following properties:
\begin{enumerate}[(i)]
\item is a symmetric, non-negative function; 
\item is harmonic in each component; 
\item and vanishes continuously on $\partial O$.
\end{enumerate}
\end{lemma}

\begin{proof}
For the sake of simplicity we denote $h_t(x,y)=a(t)p(t,x,y)$. Let $t_k\to\infty$ be a sequence such that $h_{t_k}$ converges uniformly on compact sets of $O$ to a function $h$. It is clear that $h$ is symmetric and non-negative. Notice that for any $s\in\RR$, the sequence $h_{t_k+s}$ also converges uniformly on compact sets of $O$. This is direct from Lemma \ref{le:cms99} and the hypothesis.

By the Chapman-Kolmogorov equation, for any $s\in\RR$ and large enough $k\in\NN$,
\begin{align*}
h_{t_k+s}(x,y) &= \frac{a(t_k+s)}{a(t_k)} \int_\O h_{t_k}(x,z) p(s,z,y) dz.
\end{align*}
By assumption, $h_{t_k}(x,z)\leq C_x^{1+\abs{z}}$, which is $p(s,z,y)dz$-integrable as it can be checked by comparing $p$ with the free Brownian motion's kernel. Thus, we can apply the  \dct\ to obtain
\begin{align*}
h(x,y) &= \int_\O h(x,z) p(s,z,y) dz = \EE_y(h(x,X_s)).
\end{align*}
It is standard to show that $h(x,X_s)$ is a martingale, from where its standard to deduce that $y\mapsto h(x,y)$ is harmonic by means of the optional sampling theorem. 

Consider a sequence $y_n\in O$, with $y_n\to y\in\partial O$. By using once again the Gaussian upper bound on $p$, and applying the \dct\ to $h(x,z)p(1,z,y_n)$, it is deduced  that $h(x,\cdot)$ vanishes continuously on $\partial O$.

%
\end{proof}

\begin{lemma}
\label{le:strongmarkov}
Let $U$ and $O$ be domains in $\RR^n$ that are regular for the Dirichlet problem. For $\xi\in\partial U$, $x\in U$
\begin{align}
\label{eq:partialdensity}
\PP_x(B_{T^U}\in \sigma(d\xi),T^U\in ds) = \frac12 \partial_np^U(s,x,\xi)\sigma(d\xi)ds.
\end{align}
Here, $\partial_n$ represents the inward normal derivative at $\xi\in\partial U$.

Also, if $U\sub O$, then
\begin{align}
\label{eq:strongmarkov}
p^O(t,z,y) &= p^{U}(t,z,y) + \int_0^t\int_{\partial U} \frac12 \partial_np^U(s,x,\xi)p^O(t-s,\xi, y) \sigma(d\xi)ds.
\end{align}
\end{lemma}
\begin{proof} These results are well known so we only are going to comment their proofs. The proof of \eqref{eq:partialdensity} uses Green's theorem and the heat equation, and it is very straightforward carry out. Equation \eqref{eq:strongmarkov} follows  as an elementary application of the strong Markov property at time $T^U$.
\end{proof}

The following lemma characterizes all positive, harmonic functions vanishing on $\partial\O$. In other words, we characterise the Martin boundary of $\O$. We use the notation from the Introduction.

\begin{lemma}
\label{le:extensions}
Let $u_1,\ldots,u_\k$ be the minimal harmonic functions given by \eqref{eq:harmonic_extension}. For every nonnegative harmonic function $u$ in $\O$, vanishing continuously on $\partial\O$, there are unique nonnegative coefficients $\gamma_1,\cdots, \gamma_\k$ such that
\begin{align}
\label{eq:uniqueharmonic}
u(x) = \sum_{j=1}^\k \gamma_j u_j(x),\qquad x\in\O.
\end{align}
\end{lemma}
\begin{proof}
For $x\in\O_j$, consider the harmonic function $\tilde{w}_j (x) = u(x)-\EE_x(u(B_{T^j}))$. It is standard to check that $\tilde{w}_j$ is harmonic in $\O_j$, and that vanishes continuously on $\partial\O_j$. For $m>R$, let $T^j_m$ be the exit time from the set $\O_j\cap B(a_j,m)$. By It\^o's formula, the process $u(B_{t\wedge T^j_m})$ is a bounded martingale under $\PP_x$, for $x\in\O_j$. Therefore,
\begin{align*}
u(x) &= \EE_x\rpar{u(B_{T^j_m})} = \EE_x\rpar{u(B_{T^j_m})\ind_{\kpar{T^j_m < T^j}}}  +   \EE_x\rpar{u(B_{T^j})\ind_{\kpar{T^j_m=T^j}})} \\
&\geq \EE_x\rpar{u(B_{T^j})} - \EE_x\rpar{u(B_{T^j})\ind_{\kpar{T^j_m<T^j}})}.
\end{align*}
Since $T^j_m\nearrow T^j$, monotone convergence shows that $u(x)\geq \EE_x\rpar{u(B_{T^j})}$, that is, $\tilde w_j$ is nonnegative. Thus, $\tilde w_j(x) = \gamma_j w_j(x)$ by uniqueness. For $z\in\O$, set
\begin{align*}
\tilde u(z) &= \sum_{j=1}^\k \gamma_j u_j(z) - u(z), 
\end{align*}
which is harmonic in $\O$, and vanishes continuously on $\partial O$. We will next show that $\tilde u$ is bounded, for which it is enough to show that it is bounded in each branch of $\O$.

Fix $i\in\kpar{1,\ldots,\k}$, and consider $x\in \O_i$. We have
\begin{align*}
\tilde u(x) &= 
-\EE_x\rpar{ u(B_{T^i}) } +\gamma_i\EE_x\rpar{ u_i(B_{T^i}) } + \sum_{j=1,j\neq i}^\k \gamma_j u_j(x) \\
\end{align*}
The first term on the right hand side is bounded by $\sup_{x\in\Gamma_i}\abs{u(x)}$, and the second one by $\gamma_i \sup_{x\in\Gamma_i}\abs{u_i(x)}$. The summation is bounded as each term $u_j(x)$ is bounded in $\O_i$. We conclude that $\tilde u$ is harmonic and bounded in $\O$, and vanishes continuously on $\partial\O$. It follows that $\tilde u\rpar{B_{t\wedge T}}$ is a martingale, and so
$$
\tilde u(z) = \EE_z\rpar{ \tilde u\rpar{B_{t\wedge T}} }\to 0, \text{ as } t\to\infty.
$$

Uniqueness follows from the boundedness of $u_j$ in $\O\setminus \O_j$, and its unboundedness in $\O_j$.
\end{proof}

\subsection{Asymptotics in a cone with vertex}
\label{ss:vertex}

In what follows we consider a cone $V$, with opening $\d$ and vertex $a=0$, that is, $V=C(0,\d,0)$. Let $p^V$ be the heat kernel in $V$. 
Let $0<\lambda^1<\lambda^2\leq \lambda^3\leq\cdots$ be the eigenvalues of the Laplace-Beltrami operator on $\d$, with  corresponding orthonormal basis $\{m^1,m^2,m^3,\ldots\}$ of $L^2(\d,\sigma)$. We also denote by  $\alpha^i = \rpar{\lambda^i+(\frac{n}2-1)^2}^{1/2}$.

The behaviour of the heat kernel with Dirichlet boundary conditions is well known for a cone with vertex.  
The following results are taken from \cite{BaS97}.

\begin{theorem}
\label{th:hkvertex}
For $x=r\theta,y=\rho\omega\in V$, with $\theta,\omega\in \d$ and $r=\abs{x}$, $\rho=\abs{y}$, the heat kernel with Dirichlet boundary conditions in $V$ is given by:
\begin{align}
\label{eq:hkvertex}
p^V(t,x,y) = \frac{\exp\rpar{-\frac{r^2+\rho^2}{2t}}} {t\rpar{r\rho}^{\frac{n}2 -1}} 
\sum_{i=1}^\infty J_{\alpha^i} \rpar{ \frac{r\rho}{t} } m^i(\theta) m^i(\omega)
\end{align}
where $J_\nu$ is the modified Bessel function of first kind of order $\nu$, that is, the solution of 
$$
z^2 J_\nu''(z) + zJ_\nu' - (z^2+\nu^2)J_\nu=0,
$$
satisfying the growing conditions:
\begin{align}
\label{eq:besselbound}
\frac{z^\nu}{2^\nu\Gamma(1+\nu)} \leq J_\nu(z) \leq \frac{z^\nu}{2^\nu\Gamma(1+\nu)} e^z,
\end{align}
for $z>0$, and $\nu\geq 0$.
\end{theorem}

Recall that the unique minimal positive harmonic function in $V$ is given by $v(x)=v(\abs{x}\theta) = \abs{x}^{\alpha^1-\rpar{\frac{n}{2}-1}}m^1(\theta)$. 

\begin{corollary}
\label{co:vertextimelimit}
For each $x,y\in V$, we have
\begin{align}
\label{eq:vertextimelimit01}
\lim_{t\to\infty} t^{1+\alpha^1} p^V(t,x,y) &= \frac{v(x)v(y)}{2^{\alpha^1} \Gamma(1+\alpha^1)} \\
\label{eq:vertexlimit02}
\lim_{t\to\infty} \frac{p^V(t,x,y)}{p^V(t,w,z)} &= \frac{v(x)v(y)}{v(w)v(z)}
\end{align}
Both limits are uniform in compact sets.
\end{corollary}
\begin{proof}
Clearly, \eqref{eq:vertexlimit02} follows from \eqref{eq:vertextimelimit01}, so we only prove the latter. From Theorem \ref{th:hkvertex}, we get the bound \small
\begin{align*}
\abs{ t^{1+\alpha^1}p^V(t,x,y) - \frac{ v(x)v(y) }{2^{\alpha^1}\Gamma(1+\alpha^1)} } &\leq C
\abs{ \frac{t^{\alpha^1} e^{-\frac{r^2+s^2}{2t}} }{(rs)^{\frac{n}2-1}} J_{\alpha^1}\rpar{\frac{rs}{t}} - \frac{ (rs)^{\alpha^1-(\frac{n}2-1)} }{ {2^{\alpha^1}\Gamma(1+\alpha^1)} } }  +  \\
&\qquad + {t^{-(\alpha^2-\alpha^1)}} \sum_{k=2}^\infty \frac{ (rs)^{\alpha_k-(\frac{n}2-1)} }{2^{\alpha_k}\Gamma(1+\alpha_k)},
\end{align*}\normalsize
where $C=\sup_{\theta\in \d} m^1(\theta)^2$. The uniform convergence on compact sets for the first term is easily deduced from \eqref{eq:besselbound}. The series on the right hand side converges uniformly in compact sets, so the whole term converges to zero, as $t\to\infty$, since $\alpha^2>\alpha^1$. 
\end{proof}

\section{Asymptotics in a truncated cone}
\label{se:truncated}

The main goal of this section is to extend Corollary  \ref{co:vertextimelimit} to a truncated cone $\C = C(a,\d,R)$. As before, we assume that $R=1$ and $a=0$.

We will often use the following version of the Harnack inequality up to the boundary.

%

\begin{theorem}[From \cite{Sal81}, see also \cite{Mos64}]
\label{th:harnackboundary}
Let $O$ be a precompact, regular domain for the Dirichlet problem, and let $u\geq 0$ be a solution of the heat equation on $O\times [0,T)$ with Dirichlet boundary condition. Then, given $x\in O$, there is $C_1>0$ such that $u(t,z)\leq C_1 u(T,x)$, for all $(t,z)\in[0,T)\times\overline O$ where the constant $C_1$ depends only on $x$ and $T-t$.
\end{theorem}

\begin{corollary}
\label{co:expharnackcone}
Let $V$ be a cone with vertex. For any $x\in V$ there is a constant $C_x>0$, only dependent on $x$, such that for all $y\in V$ the following inequality holds for all $t>1$:
\begin{align}
\label{eq:expharnackcone}
p^V(t,x,y) \leq C_x^{1+\abs{y}}p^V(t,x,x).
\end{align}
\end{corollary}
\begin{proof}
Assume $\abs{x}=1$, otherwise the corollary follows by scaling. The inequality holds for small $\abs{y}$, by a direct application of the boundary Harnack inequality (Theorem \ref{th:harnackboundary}), so we assume that $\abs{y} > 2$.

Let $r$ be positive, but small enough so that $B\rpar{x,r}\sub V$. It follows by scaling that $B\rpar{\nu x,r}\sub V$ for all $\nu\geq 1$. Thus, applying the standard parabolic Harnack inequality several times in the ball $B(0,r)$ to the function $u(s,z) = p^V(t+s, \nu x+z, y)$ for fixed, but arbitrary $\nu>1,y\in V$, we get
$$
p^V(t,\nu x,y)\leq C_2^{1+r\nu} p^V(t+1+r\nu,x,y)\leq C_2^{2+2r\nu}p^V(t+2+2r\nu,\nu x,y),
$$ 
for a positive constant $C_2$ that only depends on $x$. 

The heat kernel in $V$ has the following scaling property: 
$$
p^V(t,x,y) = \lambda^{-n} p^V\rpar{\frac{t}{\lambda^2},\frac{x}{\lambda},\frac{y}{\lambda}},\qquad \lambda> 0.
$$

From all the inequalities above, it follows that
\begin{align*}
p^V(t,x,y) &\leq C_3^{1+\abs{y}} p^V(t+1+r\abs{y},x\abs{y},y) \\
&= C_3^{1+\abs{y}} \abs{y}^{-n} p^V\rpar{ \frac{t+1+r\abs{y}}{\abs{y}^2}, x, \frac{y}{\abs{y}}} \\
&\leq C_1 C_3^{1+\abs{y}} \abs{y}^{-n} p^V\rpar{ \frac{t+1+r\abs{y}}{\abs{y}^2}+1, x, x}\\
&\leq C_1 C_3^{1+\abs{y}} \abs{y}^{-n} p^V\rpar{ \frac{t}{\abs{y}^2}, x, x},
\end{align*}
where the second to last line comes from the boundary Harnack inequality, whereas the last one comes form the fact that $t\mapsto p^V(t,x,x)$ is decreasing (see Lemma \ref{le:cms99}). Applying scaling once again, 
\begin{align*}
p^V(t,x,y) &\leq C_1 C_3^{1+\abs{y}} p^V\rpar{ t, x\abs{y}, x\abs{y}} \\
&\leq C_1 C_3^{3+3\abs{y}} p^V\rpar{ t, x, x},
\end{align*}
as desired.
\end{proof}

\begin{lemma}
\label{le:normalderivative}
Let $\C=C(a,\d,1)$, $\bd=a+\d$, and $V=C(a,\d,0)$. There is a universal constant $Q>0$ such that for any $x\in\C$, and $\xi \in\bd$,
\begin{align}
\limsup_{t\to \infty} \frac{\partial_n p^\C(t,x,\xi)}{p^V(t,x,x)} \leq  \frac{Q}{v(x)},
\end{align}
where $v$ is the unique minimal harmonic function in $V$, normalized as in \eqref{eq:coneharmonic}.
\end{lemma}
\begin{proof}
By a translation of coordinates, we can assume $a=0$. Set $U=B(0,1)^c$. By monotonicity of domains, and since both $p^\C(t,x,\cdot)$ and $p^U(t,x,\cdot)$ vanish on $\bd$, we have that $0\leq \partial_n p^\C(t,x,\xi) \leq \partial_n p^U(t,x,\xi)$. Recall that there are constants $A>0,B>0$ such that $\partial_n p^U(1,x,\xi)\leq A\exp(-B\abs{x}^2)$, so $ 0\leq \partial_n p^\C(1,x,\xi) \leq A\exp(-B\abs{x}^2)$ for $\xi\in \bd$. These bounds allows us to compute the normal derivative from  the Chapman-Kolmogorov equation as follows
\begin{align}
\label{eq:normalderivativechapman}
\partial_n p^\C(t+1,x,\xi) &= \int_{\C} p^\C(t,x,z) \partial_np^\C(1,z,\xi)dz \\
\nonumber &\leq \int_{\C} p^V(t,x,z) \partial_np^\C(1,z,\xi)dz.
\end{align}
Thus,
\begin{align*}
\frac{\partial_n p^\C(t+1,x,\xi)}{p^V(t,x,x)} &\leq \int_{\C} \frac{p^V(t,x,z) }{p^V(t,x,x)} \partial_np^\C(1,z,\xi)dz.
\end{align*}
We intend to apply the \dct\ to the integral on the right hand side. Equation \eqref{eq:vertexlimit02} shows pointwise convergence as $t\to\infty$, and Corollary \ref{co:expharnackcone} together with the remarks at the beginning of this proof show that  the integrand is dominated. Therefore
\begin{align*}
\limsup_{t\to \infty} \frac{\partial_n p^\C(t+1,x,\xi)}{p^V(t,x,x)} &\leq \frac{1}{v(x)} \int_{\C} v(z) \partial_np^\C(1,z,\xi)dz.
\end{align*}
The integral can be estimated using the explicit formula for $v(z)$, and the bound for $\partial_np^\C(1,z,\xi)$ discussed at the beginning of this proof. Finally, we use Lemma \ref{le:cms99} to conclude.

%
%
%
%
%
\end{proof}

\begin{theorem}
\label{th:vertextruncated}
Let $V$ a cone with opening $\d$ and vertex $a$, and let its truncated version be $\C=C(a,\d,1)$. Let $w$  be the unique minimal positive harmonic function in $\C$. Then, for all $x,y\in \C$,
\begin{align}
\label{eq:vertextruncated}
\lim_{t\to\infty} t^{1+\alpha^1} p^\C(t,x,y)  = \frac{ w(x)w(y)}{2^{\alpha^1}\Gamma(1+\alpha^1)},
\end{align}
where $\alpha^1$ is the character of $\C$. The limit is in the sense of uniform convergence on compact sets.
\end{theorem}
\begin{proof} The proof relies on equation \eqref{eq:strongmarkov} and the \dct. Let
\begin{align*}
I_1(t) &= \frac12 \int_0^{t/2}\int_{\bd} \partial_n p^\C(s,x,\xi) p^V(t-s,\xi,y) \sigma(d\xi) ds \\
I_2(t) &= \frac12 \int_{t/2}^t\int_{\bd} \partial_n p^\C(s,x,\xi)  p^V(t-s,\xi,y) \sigma(d\xi) ds.
\end{align*}
Then, equation \eqref{eq:strongmarkov} and Theorem \ref{th:hkvertex} yield,
\begin{align}
\label{eq:truncated01}
\frac{v(x)v(y)}{2^{\alpha^1}\Gamma(1+\alpha^1)} &\leq \limsup_{t\to\infty} t^{1+\alpha^1} p^\C(t,x,y) + \limsup_{t\to\infty} t^{1+\alpha^1} \rpar{I_1(t) + I_2(t)}.
\end{align}

We start by studying $I_1(t)$.  For $0\leq s\leq t/2$, Theorem \ref{th:hkvertex} shows that $t^{1+\alpha^1}p^V(t-s,\xi,y)$ converges to $\frac{ v(\xi)v(y)}{2^{\alpha^1}\Gamma(1+\alpha^1)}$. Besides, by using  Corollary \ref{co:expharnackcone} we get the following bound:
\begin{align*}
t^{1+\alpha^1} p^V(t-s,\xi,y) &= 2^{1+\alpha^1} (t/2)^{1+\alpha^1} p^V(t-s,\xi,y) \\
&\leq 2^{1+\alpha^1}  C_1C_2^{1+\abs{y}}(t/2)^{1+\alpha^1}p^V(t-s,x,x) \\
&\leq 2^{1+\alpha^1} C_1C_2^{1+\abs{y}} \rpar{t/2}^{1+\alpha^1} p^V(t/2,x,x).
\end{align*}
The right hand side is uniformly bounded for $t>1$, so the \dct\ applies:\small
\begin{align*}
\lim_{t\to\infty} t^{1+\alpha^1} I_1(t) &= \frac12 \int_0^\infty\!\!\! \int_{\bd} \partial_n p^\C(s,x,\xi) \frac{v(\xi) v(y) }{2^{\alpha^1}\Gamma(1+\alpha^1)} \sigma(d\xi)ds = \frac{ v(y) \EE_x(v(B_{T^\C})) }{2^{\alpha^1}\Gamma(1+\alpha^1)}.
\end{align*}\normalsize

Next, we study the asymptotics of $I_2(t)$. Since we don't have sharp asymptotics for $\partial_n p^\C$ yet, we are not able to use the \dct. Instead, we will resort to Fatou's lemma. For $t/2\leq s\leq t$, Lemma \ref{le:normalderivative} an Theorem \eqref{th:hkvertex} imply that 
$$
\limsup_{t\to\infty} t^{1+\alpha^1}\partial_n p^\C(s,x,\xi) \leq \frac{Q_1 v(x)}{2^{\alpha^1}\Gamma(1+\alpha^1)},
$$
where $Q_1$ only depends on $x$. To show domination, we combine equations \eqref{eq:expharnackcone}  and  \eqref{eq:normalderivativechapman} to get the bound
\begin{align*}
t^{1+\alpha^1}\partial_n p^\C(s,x,\xi) &\leq t^{1+\alpha^1} \int_{\C} p^V(s-1,x,z) \partial_np^\C(1,z,\xi)dz \\
&\leq t^{1+\alpha^1} p^V(t/2-1,x,x) \int_{\C} C_2^{1+\abs{z} } \partial_np^\C(1,z,\xi)dz.
\end{align*}
The right hand side is uniformly bounded in $t>2$ by a constant $Q_2$ that only depends on $x$. It follows that
\begin{align*}
\limsup_{t\to\infty} t^{1+\alpha^1} I_2(t) &\leq \frac{Q_1 v(x)}{2^{\alpha^1}\Gamma(1+\alpha^1)} \int_0^\infty\int_{\bd} p^V(s,\xi,y) \sigma(d\xi)ds \\
&= \frac{Q_1 v(x)}{2^{\alpha^1}\Gamma(1+\alpha^1)}  G^V(\bd,y).
\end{align*}

Using these two estimates in equation \eqref{eq:truncated01} we obtain,\small
\begin{align*}
\frac{v(x)v(y)}{2^{\alpha^1} \Gamma(1+\alpha^1)} \leq  \limsup_{t\to\infty} t^{1+\alpha^1} p^\C(t,x,y) + \frac{v(y) \EE_x(v(B_{T^\C})) + Q_1 v(x)G^V(\Gamma,y)}{2^{\alpha^1} \Gamma(1+\alpha^1)}.
\end{align*}\normalsize
Recall that $w(x) = v(x) - \EE_x(v(B_{T^\C}))$ for $x\in\C$. Thus,
\begin{align}
\label{eq:truncated02}
\frac{w(x)v(y)}{2^{\alpha^1} \Gamma(1+\alpha^1)} \leq  \limsup_{t\to\infty} t^{1+\alpha^1} p^\C(t,x,y) + \frac{ Q_1 v(x)G^V(\Gamma,y)}{2^{\alpha^1} \Gamma(1+\alpha^1)}.
\end{align}

We will deduce the theorem from this last estimate. By Corollary \ref{co:expharnackcone}, it is possible to apply Lemma \ref{le:limitharmonic} to $t^{1+\alpha^1}p^\C(t,x,y)$. Therefore, any limit point of this family has the form $\eta w(x)w(y)$ for some $\eta\geq 0$. Different limit points will correspond to different values of $\eta$. We will show that this is not the case: by monotonicity of domains, $\eta w(x)w(y) \leq \frac{v(x)v(y)}{2^{\alpha^1}\Gamma(1+\alpha^1)}$, and since $w(z)$ and $v(z)$ have the same asymptotic behavior for radially convergent $z\to\infty$, we deduce that $\eta\leq \frac{1}{2^{\alpha^1}\Gamma(1+\alpha^1)}$. Let $\eta^*=\sup\eta$, where the supremum is taken over all possible limit points. Equation \eqref{eq:truncated02} then yields
\begin{align*}
\frac{w(x)v(y)}{2^{\alpha^1} \Gamma(1+\alpha^1)} \leq  \eta^* w(x)w(y) + \frac{ Q_1 v(x)G^V(\Gamma,y)}{2^{\alpha^1} \Gamma(1+\alpha^1)}.
\end{align*}
Dividing this equation by $v(y)$ and taking $y$ radially to infinity, the second term on the right hand side vanishes in the limit as $G^V(\Gamma,y)$ is bounded for $y$ away form $\Gamma$. We obtain that $\eta^*\geq \frac{1}{2^{\alpha^1}\Gamma(1+\alpha^1)}$, which shows that the only  possible limit point is the one given by \eqref{eq:vertextruncated}. Uniform convergence on compact sets follows from Lemma \ref{le:cms00}.
\end{proof}

\section{Asymtotics in multicone domains}
\label{se:multicone}

We start by fixing $x_0\in \O$, and a sequence $(t_k)$ such that 
$$
F(x,y) = \lim_{k\to\infty} \frac{p(t_k,x,y)}{p(t_k,x_0,x_0)} =\sum_{i,j=1}^k \gamma_{ij}u_i(x)u_j(y),
$$ 
where the converges is uniform in compact sets of $\O\times\O$, and $u_j$ are the minimal harmonic functions in $\O$. This is obtained by a double application of Lemma \ref{le:extensions}. The coefficients $\gamma_{ij}\geq 0$ might depend on the sequence $(t_k)$. Notice that $F(x_0,x_0)=1$.

By passing  to  subsequences of $(t_k)$, we can assume that for all $j=1,\ldots,k$ we have that 
$$
F_j(x,y) = \lim_{k\to\infty} \frac{p^j(t_k,x,y)}{p(t_k,x_0,x_0)} =\mu_j w_j(x)w_j(y)
$$
is well defined. The convergence is uniform in compact sets of $\O_j\times\O_j$. The coefficient $\mu_j\geq 0$ may also depend on the sequence $(t_k)$. 

Our goal is to compute explicitly the coefficients $\gamma_{ij}$. In order to do this, we will use equation \eqref{eq:strongmarkov} with $O=\O$, and $U=\O_j$, and estimate the integral involved in \eqref{eq:strongmarkov}. 

It will be convenient to  fix points $\xi_j\in \bd_j$, and $z_j\in \O_j$. For $x\in\O_j$, $y\in\O$, $j=1,\ldots,\k$, we define the following object
\begin{align}
\label{eq:integral}
I^{x,y}_j(a,b;t) &= \int_a^{b}\int_{\bd_j} \PP_x(B_{T^j}\in d\xi,T^j\in du)p(t-u,\xi,y).
\end{align}
Most of the technical work of this section will be devoted to find convenient estimates for $I^{x,y}_j$.

We start  with a lemma about the function $F_j$. Recall that $\M$ denotes the set of maximal indices. 

\begin{lemma}
\label{le:conetodomain}
We have that $\mu_j=\mu$ for $j\in\M$, and $\mu_j=0$ for $j\notin\M$. Also,
there is a constant $C$, independent of the sequence $(t_k)$, such that $\mu\leq C$.
\end{lemma}
\begin{proof}
Recall that $\alpha_j$ denotes the character of the branch $\O_j$. For $j,l=1,\ldots,\k$, and points $x\in\O_j$, $y\in\O_l$ we have
$$
 \frac{p^j(t,x,x)}{p(t,x_0,x_0)} =  \frac{t^{1+\alpha_j} p^j(t,x,x)}{t^{1+\alpha_l} p^l(t,y,y)} \frac{p^l(t,y,y)} {p(t,x_0,x_0)}  \frac{1}{t^{\alpha_j-\alpha_l}}.
$$
It follows that 
$$
\mu_j w_j(x)^2 = \lim_{k\to\infty}  \frac{p^j(t_k,x,x)}{p(t_k,x_0,x_0)} = \frac{2^{\alpha_l} \Gamma(1+\alpha_l) w_j(x)^2}{2^{\alpha_j}\Gamma(1+\alpha_j) w_l(y)^2} \mu_l w_l(y)^2 \lim_{t\to\infty} \frac{1}{t^{\alpha_j - \alpha_l}},
$$
If $j\notin\M$ and $l\in\M$, we have $\alpha_l < \alpha_j$, and so $\mu_j=0$. If both $j,l\in\M$, we have $\alpha_j=\alpha_l$, and so $\mu_j=\mu_l=\mu$, only depending on $(t_k)$.

Pick any $j \in\M$. By Harnack's inequality we have
$$
\frac{p^j(t_k,z_j,z_j)}{p(t_k+2,x_0,x_0)} \leq C_H^{2}\frac{p^j(t_k,z_j,z_j)}{p(t_k,z_j,z_j)} \leq C_H^2,
$$
by monotonicity of domains. Using Lemma \ref{le:cms99}, we see that the left hand side above converges to $\mu w_j(z_j)^2$, thus,
$$
\mu \leq \frac{C_H^2}{w_j(z_j)^2} \leq \frac{C_H^2}{\inf_j w_j(z_j)^2}
$$
as desired.
\end{proof}

\begin{lemma}
\label{le:fundamental}
There is a constant $C>1$ such that, for every $M>2$, every index $1\leq j\leq \k$, $m \in\mathcal{M}$, and points $x\in\O_j,\ y\in\O$ we have for $t_k >2M+1$
\begin{align}
\label{eq:fund0M}
\limsup_k \frac{I^{x,y}_j(0,M;t_k)}{p(t_k,x_0,x_0)} &\leq C\PP_x\rpar{ B_{T^j}\in\bd_j}F(\xi_j,y),\\
\label{eq:fundtMt}
\limsup_k \frac{I^{x,y}_j(t_k-M,t_k;t_k)}{p(t_k,x_0,x_0)} &\leq C\ind_{\M}(j) G(\bd_j,y) w_j(x) w_j(z_j),\\
\label{eq:fundMtM}
\limsup_L \limsup_k \frac{I^{x,y}_j(L,t_k-L;t_k)}{p(t_k,x_0,x_0)} &\leq C \ind_{\mathcal{M}}(j) \frac{w_j(x)}{w_m(z)}F(z,y), 
\end{align}
where the last equation holds for any $z\in\O_m$. The constant $C$ depends only on the domain $\O$ and our choices of $z_j$.

\end{lemma}
\begin{proof} Take $k$ large enough such that $t_k > 2M+1$.  By the boundary Harnack inequality, there exists a positive $C_1$ such that for all $u\in [0,t_k-1]$, all $y\in \O$, and all $i=1,\ldots, \k$
\begin{align}
\label{eq:harnack}
p(t_k-u,\xi,y) \leq C_1 p(t_k-u+1,\xi_i,y).
\end{align}

By Fatou's lemma and Lemma \ref{le:cms99} we get
\begin{align*}
\varlimsup_k \frac{I^{x,y}_j(0,M;t_k)}{p(t_k,x_0,x_0)} &\leq C_1 \int_0^{M}\int_{\bd_j} \PP_x(B_{T^j}\in d\xi,T^j\in du) F(\xi_j,y),
\end{align*}
from which \eqref{eq:fund0M} follows easily.

For $u\in [0,M]$,  Lemma \ref{le:normalderivative}, Theorem \eqref{th:vertextruncated}, equation \eqref{eq:normalderivativechapman} and the \dct\ yield
\begin{align}
\label{eq:tandnormal}
\lim_{k\to\infty} \frac{\partial_n p^j(t_k-u,x,\xi)}{p^j(t_k,x,z_j)} = \frac{\partial_n w_j(\xi)}{w_j(z_j)},
\end{align}
where the convergent sequence is bounded by a constant that only depends on $M$, $x$ and $z_j$ (see Lemma \ref{le:normalderivative}). It follows by the \dct, and  Harnack's inequality, that \small
\begin{align*}
\varlimsup_k \frac{I^{x,y}_j(t_k-M,t_k;t_k)}{p(t_k,x_0,x_0)} &\leq \varlimsup_k \frac{ p^j(t_k,x,z_j) }{ p(t_k,x_0,x_0) } \int_0^{M}\int_{\bd_j} \frac12 \frac{\partial_n w_j(\xi)}{w_j(z_j)} p(u,\xi,y)\sigma(d\xi)du \\
&\leq C_2 G(\bd_j,y) \varlimsup_k \frac{ p^j(t_k,x,z_j) }{ p(t_k,x_0,x_0) } ,
\end{align*}\normalsize
where $C_2 =\max\limits_{j=1.\ldots,\k} \frac{1}{ 2v_j(z_j)} \sup_{\xi\in \bd_j} {\partial_n w_j(\xi)}\sigma(\bd_j)$. This inequality and Lemma \ref{le:conetodomain} prove \eqref{eq:fundtMt}.

Recall that there is $r_0>0$ such that for all $i=1,\ldots,\k$, we have $B_{2r_0}(\xi_i)\cap\kpar{\abs{x}=1}\sub \bd_i$.  For $x\in \O_j$, $z\in\O_m$, set
$$
C^L_{jm}(x,z) = \sup_{t>L} \frac{ \int_{\bd_j} \partial_n p^j(t,x,\xi)\sigma(d\xi) }{ \int_{\bd_m\cap B_{r_0}(\xi_m)} \partial_n p^m(t,z,\xi')\sigma(d\xi') },
$$
which is finite since $m\in\M$. From Lemma \ref{le:normalderivative}, Theorem \ref{th:vertextruncated}, and the  \dct, we obtain \small
$$
\lim_{L\to\infty} C^L_{jm}(x,z) = \ind_{\mathcal{M}}(j) \frac{ w_j(x) }{w_m(z)} \frac{\int_{\bd_j} \partial_n w_j(\xi)\sigma(d\xi)}{\int_{\bd_m\cap B_{r_0}(\xi_m)} \partial_n w_m(\xi')\sigma(d\xi')} \leq C_3 \ind_{\mathcal{M}}(j) \frac{ w_j(x) }{w_m(z)} ,
$$\normalsize
for a constant $C_3>0$.

From the standard Harnack inequality we have $p(t-u+1,\xi_m,y)\leq C_4 p(t-u+2,\xi',y)$ for all $\xi'\in B_{r_0}\rpar{\xi_m}$. The previous discussion yields the following series of inequalities \small
\begin{align*}
I^{x,y}_j&(L,t_k-L;t_k) \leq C_1 \int_{L}^{t_k-L} \!\!\! {\int_{\bd_j} \!\!\! \frac{1}{2}\partial_n p^j(u,x,\xi)} p(t_k-u+1,\xi_m,y) \sigma(d\xi) du \\
&\leq C_1 C^L_{jm}(x,z)  \int_{L}^{t_k-L} {\int_{\bd_m \cap B_{r_0}(\xi_m)} \frac{1}{2}\partial_n p^m(u,z,\xi')} p(t_k-u+1,\xi_m,y) \sigma(d\xi') du \\
&\leq C_1 C_4 C^L_{jm}(x,z) \int_{L}^{t_k-L} \int_{\bd_m \cap B_{r_0}(\xi_m)}  \frac{1}{2} \partial_n p^m(u,z,\xi')p(t_k-u+2,\xi',y)\sigma(d\xi') du \\
&\leq C_1 C_4 C^L_{jm}(x,z) \int_{L}^{t_k-L} \int_{\bd_m}  \frac{1}{2} \partial_n p^m(u,z,\xi')p(t_k-u+2,\xi',y)\sigma(d\xi') du \\
&\leq C_1 C_4 C^L_{jm}(x,z) p(t_k+2,z,y).
\end{align*}\normalsize
Equation \eqref{eq:fundMtM} now follows from Lemma \ref{le:cms99}.
\end{proof}

\begin{lemma}
For the coefficients of the function $F$ defined at the beginning of this section:
\label{le:gammas}
\begin{enumerate}[(i)]
\item $\gamma_{ij}=0$ if $i\notin\M$ or $j\notin\M$.
\item There is a universal constant $C$ depending only on the domain $\O$ such that $\gamma_{ij} \leq C \gamma_{jm}$ for $i,j,m\in\M$, with $i\neq j$.
\end{enumerate}
\end{lemma}
\begin{proof}
Let $x\in\O_i$ and $y\in\O_j$. By \eqref{eq:strongmarkov},
\begin{align*}
p(t,x,y) = p^i(t,x,y)\delta_{ij} + I^{x,y}_i(0,t;t).
\end{align*}
From Lemmas \ref{le:fundamental} and \ref{le:conetodomain} we obtain for $m\in\M,\ z\in\O_m$,
{\footnotesize
\begin{align*}
F(x,y) &\leq \delta_{ij}\mu_iw_i(x)w_i(y) + C\rpar{F(\xi_i,y) + G(\bd_i,y)\ind_{\M}(i)w_i(x)w_i(z_i) + \ind_{\M}(i) \frac{w_i(x)}{w_m(z)} F(z,y)  } \\
&\leq C\ind_{\M}(i)w_i(x) \rpar{\delta_{ij}w_j(y)  + G(\bd_i,y) w_i(z_i) + \frac{F(z,y) }{w_m(z)}} +CF(\xi_i,y).
\end{align*}
}

The use of this inequality is twofold. First, if $i\notin \M$, by taking $x$ radially to infinity in $\O_i$ we find that $\gamma_{ij}u_i(x)u_j(y)\leq C F(\xi_i,y)$ is only possible if $\gamma_{ij}=0$. By symmetry of the kernel we conclude $(i)$. 


Secondly, consider  $i,j,m\in\M$, with $i\neq j$ and $z\in\O_m$. Dividing the inequality by $w_i(x)w_j(y)$, and taking $x,y,z\to\infty$ radially in their respective branches,  we obtain that $\gamma_{ij} \leq C \gamma_{mj}$, as desired.
\end{proof}

\begin{remark} 
\label{re:gammastar}
Set $\gamma^* =\max_{i\in\M} \gamma_{ii}$ and fix $m\in\M$ such that $\gamma^*=\gamma_{mm}$. The previous lemma states that $\gamma_{ij}\leq C\gamma^*$ for all $i,j\in\M$. Also, notice that 
$$
\gamma_{mm}u_m(x)u_m(y)\leq F(x,y) \leq (1+C)\gamma_{mm}\sum_{i,j\in\M} u_i(x)u_j(y).
$$
Since $F(x_0,x_0)=1$, we obtain that $\gamma^*$ is bounded below by a constant that is independent of the sequence $(t_k)$.

It follows that 
\begin{align*}
\lim_k \frac{p(t_k,x,y)}{p(t_k,z,z')} &= \frac{F(x,y)}{F(z,z')} \leq (1+C) \sum_{i,j\in\M} \frac{u_i(x)u_j(y)}{u_m(z)u_m(z')}
\end{align*}
where the constant $C$ comes from Lemma \ref{le:fundamental}. In particular, if $x\in\O_m$, the inequality\small
\begin{align}
\limsup_{t\to\infty} \frac{p(t,x,\xi_m)}{p(t,x,x)} &\leq \frac{1+C}{u_m(x)} \sum_{j\in\M} u_j(\xi_m)  \rpar{1+ \frac{1}{u_m(x)} \sup_{z\in \O_m}\sum_{i\neq m} u_i(z) }
\end{align}\normalsize
holds. 
If we fix $\widehat{x}\in \d_m$, and let $r>0$ be suficiently large,  for $x=a_m+r\widehat{x}$ this inequality implies that  
\begin{align}
\label{eq:choiceofx}
\limsup_{t\to\infty} \frac{p(t,x,x_0)}{p(t,x,x)} &\leq \frac{C_5}{u_m(x)},
\end{align}
where $C_5>0$ is independent of $r$.
\end{remark}

\begin{lemma}
\label{th:limitsintexist}
The following inequalities hold
$$
0<\liminf t^{1+\alpha} p(t,x,y),\quad \limsup t^{1+\alpha} p(t,x,y)<\infty.
$$
\end{lemma}
\begin{proof}
\newcommand{\tha}{\theta}
The first inequality is direct from monotonicity of domains and Theorem \ref{th:vertextruncated} applied to $\O_m\sub\O$.

For the second one, notice that by Harnack's inequaliy, it suffices to prove the theorem for $x=y\in\O_m$. We start by setting some constants that will be relevant to our estimates:  fix $\widehat{x}\in \d_m$, and consider $x=a_m+r\widehat{x}$. Then, the Harnack constant
$$
C_H = \sup_{s>1,\xi\in \bd_m} \frac{p(s,\xi,r\widehat{x})}{p(s+1,\xi_m,r\widehat{x})},
$$
 is independent of $r>1$. 
 
Fix $0<\theta<1$. In view of \eqref{eq:choiceofx}, we can find $x\in\O_m$ such that 
$$
 C_H 2^{1+\alpha} \limsup_{t\to \infty} \frac{ p(t,\xi_m, x) }{ p(t,x, x)} \leq C_H 2^{1+\alpha} \frac{C_5}{u_m(x)} = \frac{\theta}2.
$$
We fix such an $x=a_m+r\widehat{x}$ from now on. It follows that for large enough $t_0$, the inequality
\begin{align}
\label{eq:theta}
\frac{p(t,\xi_m,x)}{p(t,x,x)} \leq \frac{\theta}{2^{1+\alpha}C_H},\qquad \forall\ t>t_0
\end{align}
holds.


Next we get estimates using technics somewhat similar to the ones we have used in Lemma \ref{le:fundamental}. Let $t>2t_0$. By Harnack's inequality, and since $t\mapsto p(t,x,x)$ is decreasing
\begin{align*}
I^{x,x}_m(0,t/2;t) & \leq C_H \int_0^{t/2} \PP_x(B_{T^m}\in \bd_m,T^m\in du) p(t-u+1,\xi_m,x) du \\
&\leq \frac{\theta}{2^{1+\alpha}}  \PP_x(B_{T^m}\in \bd_m, T^m<t/2) p(t/2,x,x) \\
&\leq \frac{\theta}{2^{1+\alpha}}  p(t/2,x,x).
\end{align*}
It follows that for all $t>2t_0$,
$$
t^{1+\alpha} I^{x,x}_m(0,t/2;t) \leq \theta \rpar{\frac{t}2}^{1+\alpha} p\rpar{t/2,x,x}.
$$
On the other hand,
\begin{align*}
t^{1+\alpha}I^{x,x}_m(t/2,t;t) & \leq \int_0^{t/2}\int_{\bd_m} \frac{t^{1+\alpha}}{2} \partial_n p^m(t-u,x,\xi) p(u,\xi,x) \sigma(d\xi)du.
\end{align*}
The right hand side converges by \eqref{eq:tandnormal}, Theorem \ref{th:vertextruncated}, and the \dct,  to
\begin{align*}
C_2 = \frac{w_m(x)}{2^{1+\alpha}\Gamma(1+\alpha)} \int_0^\infty\int_{\bd_m} {\partial_n w_m(\xi)}{} p(u,\xi,x) \sigma(d\xi)du &<\infty.
\end{align*}

Putting together these two estimates, we have that, for the continuous function $\varphi(t)=t^{1+\alpha}p(t,x,x)$, $t\geq 2t_0$,  it holds that
$$
\varphi(t) \leq C_3 + \theta\varphi(t/2),\qquad t>2t_0,
$$
where $C_3 = C_2 + \frac{w_m(x)^2}{2^\alpha\Gamma(1+\alpha)}$.
By iteration of the inequality above, it is easy to deduce that, if $t/2^N\in [2 t_0,4t_0]$ then
$$
\varphi(t) \leq C_3\sum_{k=0}^{N-1}\theta^k + \theta^N \varphi(t/2^N)\leq \frac{C_3}{1-\theta} + \sup_{s\in [2t_0,4t_0]}\varphi(s),
$$
which finishes the proof.
\end{proof}

\subsection{Proof of Theorem \ref{th:mainkernel}}
The proof is reminiscent of the one we gave for Theorem \ref{th:vertextruncated}. For fixed $x,y\in\O$, Lemma \ref{th:limitsintexist} ensures that $t^{1+\alpha}p(t,x,y)$ is bounded. Harnack's inequality then ensures that the same holds for $x,y$ in any compact set of $\O$. This shows that Lemmas \ref{le:cms00} and  \ref{le:limitharmonic} apply. Let $(t_k)$ be a sequence such that $t_k^{1+\alpha}p(t_k,\cdot,\cdot)$ converges uniformly on compact sets. The limit then has the form $H(x,y)=\sum_{i,j=1}^\k \eta_{ij}u_i(x)u_j(y)$, where $\eta_{ij}$ might depend on the sequence $(t_k)$.\\

Let $x\in\O_m$, and $y\in\O$. Set
\begin{align*}
I_1(t) &= \frac12 \int_0^{t/2}\int_{\bd_m} \partial_n p^m(s,x,\xi) p(t-s,\xi,y) \sigma(d\xi) ds, \\
I_2(t) &= \frac12 \int_{t/2}^t\int_{\bd_m} \partial_n p^m(s,x,\xi)  p(t-s,\xi,y) \sigma(d\xi) ds.
\end{align*}
An application of the \dct, as in the proof of Lemma \ref{th:limitsintexist}, yields that
\begin{align*}
\lim_{k\to\infty} t_k^{1+\alpha}I_1(t_k) &= \frac12 \int_0^{\infty}\int_{\bd_m} \partial_n p^m(s,x,\xi) H(\xi,y) \sigma(d\xi) ds \\
&= \EE_x\rpar{H\rpar{B_{T^m},y}}.
\end{align*}
On the other hand, by using equations \eqref{eq:harmonic_extension} and \eqref{eq:harmonic_extension_2} 
\begin{align*}
\lim_{k\to\infty} t_k^{1+\alpha}I_2(t_k) &= \ind_\M(m)\frac12 \int_0^{\infty} \int_{\bd_m} w_m(x)\partial_n w_m(\xi)  p(t-s,\xi,y) \sigma(d\xi) ds \\
&= \ind_\M(m)\frac{w_m(x)\rpar{u_m(y)-w_m(y)}}{2^\alpha \Gamma(1+\alpha)}.
\end{align*}

Using these two estimates, and \eqref{eq:strongmarkov}, we arrive to
\begin{align*}
H(x,y) = \EE_x\rpar{H\rpar{B_{T^m},y}} + \ind_\M(m)\frac{w_m(x)u_m(y)}{2^\alpha \Gamma(1+\alpha)},\qquad x\in\O_m, y\in\O.
\end{align*}
Recall that $\EE_x\rpar{H\rpar{B_{T^m},y}}$ is bounded as function of $x$. By taking the limit of $H(x,y)/u_m(x)$, with $x$ going radially to infinity, we find that 
$$
\sum_{j=1}^\k \eta_{mj}u_j(y) = \ind_\M(m) \frac{u_m(y)}{2^\alpha \Gamma(1+\alpha)},
$$
By uniqueness of the decomposition \eqref{eq:uniqueharmonic}, we find that the only nonzero coefficients are $\gamma_{mm} = \frac{1}{2^\alpha \Gamma(1+\alpha)}$ for $m\in\M$. This shows \eqref{eq:mainkernel}. Uniform convergence on compact sets is direct from Lemma \ref{le:cms00}.
\qed

The following corollary is a direct consecuence of the previous theorem.

\begin{corollary}
\label{co:ratiokernel}
Let $\O$ be a mutlticone domain, with maximal index set $\M$. Then,
\begin{align}
\label{eq:limit}
\lim_{t\to\infty} \frac{ p(t,x,y) }{p(t,w,z)} = \frac{\sum_{j\in\M} u_j(x)u_j(y)}{\sum_{j\in\M} u_j(w)u_j(z)}.
\end{align}
The convergence is uniform in compact sets.
\end{corollary}

\section{Asymptotics for the exit time}
\label{se:exit}

The following result is taken form \cite{BaS97}.

\begin{theorem}
\label{th:smits}
Let $V\sub\RR^n$ be a cone with vertex $0$ and opening $\d$. Assume that $\d$ is regular por the Laplace-Beltrami operator on $\SS^{n-1}$, and let $\alpha$ be the character of $\d$. Set $\kappa=1+\alpha-n/2$, and let $T^V$ be the Brownian exit time from $V$. Then, for each $x\in V$,
\begin{align}
\label{eq:smits01}
\lim_{t\to\infty} t^{\kappa/2}\PP_x(T^V>t) = \gamma_V v(x).
\end{align}
Here $v(x)=\abs{x}^{\kappa} m^1(x/\abs{x})$ is the harmonic function defined in \eqref{eq:coneharmonic}, where $m^1$ is the only non-negative eigenfunction of the Laplace-Beltrami operator on $\d$ with Dirichlet boundary conditions. Also,
$$
\gamma_V = \frac{\Gamma \rpar{\frac{\kappa+n}2}}{{2}^{\kappa/2}\Gamma\rpar{\kappa + \frac{n}2}}\int_\d m^1(\theta) \sigma(d\theta).
$$
\end{theorem}

\begin{remark}
From this theorem, the scaling property of the heat kernel in $V$, and Harnack's inequality up to the boundary, we get the following  bound:
\begin{align*}
t^{\kappa/2}\PP_x(T^V>t)&=t^{\kappa/2}\PP_{x/\abs{x}} \rpar{T^V> \frac{t}{\abs{x}^2}} \leq C_H t^{\kappa/2}\PP_{\xi} \rpar{T^V>\frac{t}{\abs{x}^2}}\\
&\leq C_H\abs{x}^\kappa \rpar{ \frac{t}{\abs{x}^2} }^{\kappa/2} \PP_{\xi} \rpar{T^V>\frac{t}{\abs{x}^2}}.
\end{align*}
where $\xi\in \d$ is fixed. 
Using \eqref{eq:smits01}, we can pick $t_\xi$ such that whenever $t/\abs{x}^2 > t_\xi$, the right hand side of the last display is bounded by $C\abs{x}^\kappa$, where $C$ depends on our choice of $\xi$.
%
For $t /  \abs{x}^2 \leq t_\xi$, we have that $t^{\kappa/2}\leq t_\xi^{\kappa/2} \abs{x}^\kappa$. We deduce that there is a universal constant $C>0$ such that
\begin{align}
\label{eq:bound01}
t^{\kappa/2}\PP_x(T^V>t) \leq C \abs{x}^\kappa,\qquad x\in V,\ t>0.
\end{align}
By monotoncity of domains, the same inequality holds for $T^\C$, and $x\in \C$, where $\C$ is any truncated cone with opening $\d$.
\end{remark}

The following lemma will be the key tool when extending the previous result to multicones.
\begin{lemma}
\label{le:time_formula}
Let $T$ be the exit time from a multicone set $\O$, let $T^j$ be the exit time from $\O_j$, and pick $x\in\O_i$ for some $i=1,\ldots,\k$. We have 
\begin{align}
\nonumber
\PP_x(T>t) &= \PP_x(T^i>t) + \PP_x(B_t\in\O_0, T>t) + \\
\label{eq:time_formula}
&\qquad + \frac12\sum_{j=1}^\k \int_0^t\int_{\bd_j} \partial_n \PP_z(T^j>t-s)p(s,x,z)\sigma(dz)ds.
\end{align}
\end{lemma}
\begin{proof}
For $j=1,\ldots,k$, and $0\leq s\leq t$ define the functions
$$
f_j(s) = \int_{\O_j} \PP_z(T^j>t-s) p (s,x,z) dz.
$$
For $s<t$, since $u(s,z)=\PP_z(T^j>s)$, and $v(s,z)=p(s,x,z)$ are solutions of the heat equation with dirichlet boundary condition in $\O_j$ and $\O$ respectively, we have by Green's formula
\begin{align*}
\dd{f_j(s)}{s} &= \frac12 \int_{\O_j} -p (s,x,z) \Delta_z \PP_z(T^j>t-s) + \PP_z(T^j>t-s) \Delta_z p (s,x,z) dz \\
&=\frac12 \int_{\bd_j} p (s,x,z) \partial_n \PP_z(T^j>t-s) - \PP_z(T^j>t-s) \partial_n p (s,x,z) \sigma(dz)\\
&=\frac12 \int_{\bd_j} p (s,x,z) \partial_n \PP_z(T^j>t-s) \sigma(dz),
\end{align*}
where, as usual, $\partial_n$ represents the (inward) normal derivative. Then, for every $\e>0$,
\begin{align}
\label{eq:fjota}
f_j(t-\e)-f_j(0) 
&= \frac12 \int_0^{t-\e} \int_{\bd_j} p (s,x,z) \partial_n \PP_z(T^j>t-s) \sigma(dz).
\end{align}
In order to extend this equation to $\e=0$, we need an estimate for $\partial_n \PP_z(T^j>u)$ for $u$ near zero. The process $B$ leaves $\O_j$ before the norm of $B$ hits level $1$. Since $\rho_t=\abs{B_t}$ is a Bessel process, if we let $\tau$ be the hitting time of $1$ for an $n-$dimensional Bessel process, then, for $t<\tau$,
$$
\rho_t = \abs{z} + \beta_t + \frac{n-1}{2}\int_0^t \rho_s^{-1} ds\leq  \abs{z} + \beta_t + \frac{n-1}{2}t,
$$
for some one dimensional Brownian motion $\beta_t$. Let $\tau^\mu$ be the hitting time of zero for the Brownian motion with drift $\beta_t+\mu t$, with $\mu=\frac{n-1}{2}$. It follows that 
$$
\PP_z(T^j>u) \leq \PP_{\abs{z}}(\tau>u)\leq \PP_r(\tau^\mu>u),\qquad r=\abs{z}-1.
$$
As $\PP_z(T^j>u)$ vanishes on $\bd_j$, we get $\partial_n \PP_z(T^j>u)\leq \partial_r\PP_r(\tau^\mu>u)\vert_{r=0}$. The distribution of $\tau^\mu$ is well known (see equation (5.12), pp 197 \cite{KaS91}):
$$
\PP_r(\tau^\mu\in dt) = \frac{r}{\sqrt{2\pi t^3}} \exp\spar{-\frac{(r+\mu t)^2}{2t}} dt,\qquad t>0.
$$
A direct computation shows that
\begin{align*}
\PP_r(\tau^\mu > u) &= \int_u^\infty \frac{r}{\sqrt{2\pi t^3}} \exp\spar{-\frac{(r+\mu t)^2}{2t}} dt, \\
\partial_r \PP_r(\tau^\mu > u)\vert_{r=0} &= \int_u^\infty \frac{1}{\sqrt{2\pi t^3}} \exp\spar{-\frac{\mu^2t}{2}} dt 
\\
&\leq \int_u^1 \frac{t^{-3/2}}{\sqrt{2\pi}} dt + \int_1^\infty \frac{1}{\sqrt{2\pi}} \exp\spar{-\frac{\mu^2 t}{2}} dt \\
&= \sqrt{\frac2{\pi}} \rpar{\frac1{\sqrt u}-1} + \sqrt{\frac{2}{\pi}}\mu^{-2}e^{-\mu^2/2},\qquad u<1,
\end{align*}
which is integrable (in $u$) near zero. Therefore, we can apply the \dct\ in \eqref{eq:fjota} and use the continuity of $f_j$ to deduce that this equation also holds for $\e=0$. Adding all the equations for $j=1,\ldots, \k$, using that $f_j(0)= \ind_{\O_j}(x)\PP_x(T^j>t)$ and $f_j(t)=\PP_x(B_t\in {\O_j},T>t)$, and adding the contribution from $\O_0$, we obtain \eqref{eq:time_formula}.
\end{proof}

\subsection{Proof of Theorem \ref{th:mainexit} for truncated cones}

Theorem \ref{th:smits} is also valid if we change the cone $V=C(a,\d,0)$ for its truncated version $\C=C(a,\d,1)$, but the limit turns out to be $\gamma_V w(x)$, where $w(\cdot)$ is the unique positive harmonic function in $\C$ that vanishes on $\partial \C$, normalized at infinity, such that $\lim_{r\to\infty} r^{-\kappa} w(a+r\theta) = 1$ for any $\theta\in \d$.  Recall that $\bd = a+\d$ is the  base of $\C$.

Indeed, notice that $(t,x)\mapsto \PP_x(T^\C>t)$ solves the heat equation in $ \C$. By Harnack's inequality, the family of functions  $h(t,x) = t^{\kappa/2}\PP_x(T^\C>t)$, indexed by $t>0$, is bounded on compact sets of $\C$. Since $\PP_x(T^\C>t)\leq \PP_x(T^V>t)$, Lemma \ref{le:limitharmonic}
applies and we conclude that any limit point has the form $\mu w(x)$. 
 Of course, the constant $\mu$ may depend on the sequence $(t_k)$ that makes $h(t_k,\cdot)$ converge. Nevertheless, we have $\mu\leq\gamma_V$ by monotonicity of domains, where $\gamma_V$ is the same constant as in Theorem \ref{th:smits}.

On the other hand, from Lemma \ref{le:time_formula},  we have that for all $x\in \C$
\begin{align*}
\PP_x(T^V>t) &= \PP_x(T^\C>t) + \PP_x(\abs{B_t}<1, T^V>t) + \\ 
&\qquad\qquad + \frac12\int_0^t\int_{\bd} \partial_n\PP_z(T^\C>t-s)p^V(s,x,z)\sigma(dz)ds.
\end{align*}
Also,  by Harnack's inequality  $t^{\kappa/2}\PP_x(\abs{B_t}<1, T^V>t) \leq C_1 t^{\kappa/2} p(t+1,x,x_0) $, which converges to zero, as $t\to\infty$ by Theorem \ref{th:mainkernel}. Thus,
\begin{align}
\label{eq:gammavmu}
\gamma_V v(x) \leq \mu w(x) + \varlimsup_{t\to\infty} \frac{t^{\kappa/2} }{2}\int_0^t\int_{\bd_j} \partial_n\PP_z(T^\C>t-s) p^V(s,x,z)\sigma(dz)ds.
\end{align}
We will next show how to control the integral on the right hand side of \eqref{eq:gammavmu}.

By applying Fubini's theorem to the Chapman-Kolmogorov equation for the heat kernel, we get for $t,s>0$,
\begin{align*}
\PP_x(T^\C>t+s) &= \int_\C p^\C(s,x,z) \PP_z(T^\C>t) dz.
\end{align*}
Using the heat kernel of the exterior of a ball we get the upper bound $\partial_x p^\C(s,x,z)\leq Ae^{-B\abs{z}^2}$ for $s>1$ and all $x\in \C$. We can apply the \dct\ to get for $x\in \bd$
\begin{align}
\label{eq:partialtime}
\partial_n\PP_x(T^\C>t+s) &= \int_\C \partial_{n} p^\C(s,x,z) \PP_z(T^\C>t) dz.
\end{align}
Moreover, using \eqref{eq:bound01}, it is easy to obtain the following limit by using again the \dct
$$
\lim_kt_k^{\kappa/2}\partial_n\PP_x(T^\C>t_k)= \mu\int_\C \partial_n p^\C(s,x,z)w(z) dz = \mu\partial_n w(x).
$$
The last equality holds because $w$ is harmonic.

As usual, we split the integral from \eqref{eq:gammavmu} into:
\begin{align*}
I_1(t) &= \int_0^{t/2}\int_{\bd} \partial_n\PP_z(T^\C>t-s) p^V(s,x,z)\sigma(dz)ds \\
&\leq C_H \int_{\bd} \partial_n\PP_z(T^\C>t/2) \sigma(dz) \int_0^{t/2}p^V(s+1,x,\xi)ds.
\end{align*}
This shows that  $\varlimsup_{t\to\infty}  t^{\alpha/2}I_1(t) \leq C_1 G^V(x,\xi)$, where $C_1>0$ is universal.

Also, by the boundary Harnack inequality
\begin{align*}
I_2(t) &= \int_{t/2}^t \int_{\bd} \partial_n\PP_z(T^\C>t-s) p^V(s,x,z)\sigma(dz)ds \\
&\leq C_H \int_{t/2}^t \int_{\bd} \partial_n\PP_z(T^\C>t-s)\sigma(dz)ds\ p^V(t/2,x,x) \\
&\leq C_x t^{-1-\alpha}\int_0^{t/2}\int_{\bd} \partial_n \PP_z(T^\C>s)\sigma(dz) ds \\
\end{align*}
where $C_x$ only depends on $x$.
Using bounds for the exit time for a Bessel process from $[1,\infty)$ as in Lemma \ref{le:time_formula}, 
we get that  $\int_0^1\partial_n\PP_z(T^\C>s)ds \leq Q$, independently of $z\in \bd$. Then
 \begin{align*}
t^{\kappa/2}I_2(t)  &\leq C_x t^{-1-\alpha+\kappa/2}\rpar{ Q\abs{\bd} + \rpar{\frac{t}2-1}\int_{\bd} \partial_n \PP_z(T^\C>1)\sigma(dz)}.
\end{align*}
It follows that $t^{\kappa/2}I_2(t)\to 0$ as $t\to\infty$. Equation \eqref{eq:gammavmu} now reads
$$
\gamma_V v(x) \leq \mu w(x) + C_1G^V(\xi,x),\qquad x\in \C.
$$
Since $G^V(\xi,x)$ remains bounded as $x\to\infty$ radially in $\C$, we deduce that $\gamma_V=\mu$, which proves the asymptotic for the survival probability.  
\qed

\subsection{Proof of Theorem \ref{th:mainexit}}

In formula \eqref{eq:time_formula}, the first term is controlled by our result from the previous section. The second term goes to zero  by using  Harnack's inequality up to the boundary, that is, for some $x_0\in\O_0$,
\begin{align*}
\PP_x(B_t\in\O_0,T>t) &= \int_{\O_0} p(t,x,z)dz \leq C_H\abs{\O_0} p(t+1,x,x_0),
\end{align*}
where $\abs{\O_0}$ stands for the Lebesgue measure of the core $\O_0$.  It follows that $t^{\kappa/2} \PP_x(B_t\in\O_0,T>t)$ converges to zero as $t\to\infty$ for each $x\in\O$.

Next, we deal with the summation terms. In order to do this, we will find limits for the following two objects:
\begin{align*}
t^{\kappa/2}I_1(t) &= t^{\kappa/2}\int_0^{t/2}\int_{\bd_j} \partial_n \PP_z(T^j>t-s) p(s,x,z) \sigma(dz)ds, \\
t^{\kappa/2}I_2(t) &= t^{\kappa/2}\int_0^{t/2}\int_{\bd_j} \partial_n \PP_z(T^j>s) p(t-s,x,z) \sigma(dz)ds. 
\end{align*}
An analogous proof as the one in the last part of the previous section, shows that $t^{\kappa/2}I_2(t)$ converges to zero.

As for $t^{\kappa/2}I_1(t)$, 
if $0\leq s\leq t/2$, our computations in the previous section show that $t^{\kappa/2} \partial_n \PP_z(T^j>t-s) $ converges to $\gamma_{j}\partial_n w_j(z)$ for $z\in \bd_j$ and $j\in\M$, otherwise, it converges to zero. Here,
$$
\gamma_j = \frac{\Gamma \rpar{\frac{\kappa+n}2}}{{2}^{\kappa/2}\Gamma\rpar{\kappa + \frac{n}2}}\int_{\bd _j}m^1_j(\theta) \sigma(d\theta).
$$
To show domination, we use monotonicity of $t\mapsto \partial_{n}\PP_x(T^j>t)$, equation \eqref{eq:partialtime}, the  bound $\partial_n p^\C(1,x,z)\leq Ae^{-B\abs{z}^2}$, and equation \eqref{eq:bound01}. We find that
\begin{align*}
t^{\kappa/2} \partial_n \PP_x(T^j>t-s) 
&\leq  C_3  \int_\C  Ae^{-B\abs{z}^2} \abs{z}^\kappa dz < \infty.
\end{align*}
Thus, by the \dct, we deduce that
\begin{align*}
\lim_{t\to\infty} t^{\kappa/2} I_1(t) &= \gamma_{j} \int_0^\infty\int_{\bd_j}  \partial_n w_j(z)p(s,x,z) \sigma(dz)ds \\
&= \gamma_{j}\int_{\bd_j} \partial_n w_j(z) G(x,z) \sigma(dz) = 2 \gamma_{j} \rpar{u_j(x) - w_j(x)},
\end{align*}
by Fubini's theorem, and equation \eqref{eq:harmonic_extension}.

%
Putting all together,  for $x\in\O$,
\begin{align*}
\lim_{t\to\infty} t^{\kappa/2} \PP_x(T>t) &= 
 \sum_{k\in\M} \gamma_{k} u_k(x),
\end{align*}
which is \eqref{eq:mainexit}.
\qed

\section{Renormalized Yaglom limit for multicones}
\label{se:yaglom}

In what follows, we set $\beta = 1+\alpha +n/2$. Notice that $\beta/2+\kappa/2=1+\alpha$, which will be conveniently used later.

From Theorem \ref{th:hkvertex}, it is straightforward to get that for $x,y\in V=C(0,\d,0)$
\begin{align}
\label{eq:yaglomV}
\lim_{t\to\infty} t^{\beta/2} p^V(t,x,\sqrt{t}y) = \frac{v(x)v(y)}{\zed} e^{-\abs{y}^2/2}.
\end{align}
The limit above holds uniformly in compact sets of $\overline{V}$.

In order to extend this result to multicones, we start with the case of a truncated cone $\C=C(0,\d,1)$.

\begin{lemma}
\label{le:yaglomC}
Let $x,y\in\C$. Then,
\begin{align}
\label{eq:yaglomC}
 \lim_{t\to\infty} t^{\beta/2} p^\C(t,x,\sqrt{t} y) = \frac{w(x)v(y)}{\zed} e^{-\abs{y}^2/2}.
 \end{align}
\end{lemma}
\begin{proof} 
As in the proof of Lemma \ref{le:normalderivative}, we have that $t^{1+\alpha}\partial_n p^\C(t,x,\xi) \leq Q_x$ for all $t>t_0$, $\xi\in\bd$. Thus, using the boundary Harnack inequality, there exist $C_x>0$ only dependent on $x$, such that
\begin{align}
\label{eq:yaglom01}
\frac{t^{\beta/2}}{2} \int_0^{t/2}\int_{\bd} \partial_n p^\C(t-s,x,\xi) p^V(s,\xi,\sqrt{t} y) d\xi ds &\leq C_x G^{V}(x,\sqrt{t}y) t^{{\beta/2}-(1+\alpha)}.
\end{align}
For large $t$, the quantity $G^V(x,\sqrt{t}y)$ is bounded, and as ${\beta/2}-(1+\alpha)=-\kappa/2<0$, we deduce that
\begin{align}
\lim_{t\to \infty} \frac{t^{\beta/2}}{2} \int_0^{t/2}\int_{\bd} \partial_n p^\C(t-s,x,\xi) p^V(s,x,\sqrt{t} y) d\xi ds = 0.
\end{align}

On the other hand, from Theorem \ref{th:hkvertex}, it is direct to find the bound
\begin{align}
\label{eq:yaglom_bound_pV}
t^{\beta/2} p^V(t,x,\sqrt{t}y) \leq C \sum_{i=1}^\infty \frac{(\abs{x}\abs{y})^{\alpha^i-\rpar{\frac{n}{2}-1}}}{2^{\alpha^i}\Gamma(1+\alpha^i)} < \infty,
\end{align}
for $t>2$ and some universal constant $C>0$. It follows by \eqref{eq:yaglomV} and the \dct\ that
\begin{align}
\nonumber
\lim_{t\to \infty} \frac{t^{\beta/2}}{2} \int_0^{t/2}\int_{\bd} &\partial_n p^\C(s,x,\xi) p^V(t-s,\xi,\sqrt{t} y) d\xi ds =  \\ 
\nonumber 
&= \frac{v(y)e^{-\abs{y}^2/2}}{\zed} \int_0^\infty\int_\bd \frac{1}{2} \partial_n p^\C(s,x,\xi)v(\xi) d\xi ds \\
&= e^{-\abs{y}^2/2} \frac{v(y)\EE_x(v(B_{T^\C}))}{\zed}
\end{align}

Plugging the last two equations into \eqref{eq:strongmarkov}, with $O=V$ and $U=\C$, we obtain
\begin{align*}
\lim_{t\to\infty} t^{\beta/2} p^\C(t,x,\sqrt{t}y) &= \lim_{t\to\infty} t^{\beta/2} p^V(t,x,\sqrt{t}y)  - e^{-\abs{y}^2/2} \frac{v(y)\EE_x(v(B_{T^\C}))}{\zed},
\end{align*}
from where \eqref{eq:yaglomC} is direct to deduce by using \eqref{eq:truncatedharmonic}.
\end{proof}

\begin{lemma}
We have for each $y\in V$
\label{le:yaglomN}
\begin{align}
\label{eq:yaglomNsup}
\sup_{t>2}\sup_{\xi\in \bd} t^{\beta/2} \partial_n p^\C(t,\sqrt{t}y,\xi) <\infty
\end{align}
Also, for $y\in \C$ and $\xi\in\bd$,
\begin{align}
\label{eq:yaglomN}
&\lim_{t\to\infty} t^{\beta/2} \partial_n p^\C(t,\sqrt{t}y,\xi) = \frac{\partial_n w(\xi) v(y)}{\zed}e^{-\abs{y}^2/2}.
\end{align}
The limit holds in the sense of uniform convergence in compact sets.
\end{lemma}
\proof
Recall that, from the bound for the heat kernel of the exterior of a ball, and monotonicity of domains\\
\begin{align}
\label{eq:yaglom_normal}
t^{\beta/2} \partial_np^\C(t+1,\sqrt{t}y,\xi) &= t^{\beta/2} \int_\C \partial_np^\C(1,z,\xi) p^\C (t,\sqrt{t}y,z) dz \\
&\leq t^{\beta/2} \int_\C Ae^{-B\abs{z}^2} p^V (t,\sqrt{t}y,z) dz \\
&\leq AC\int_\C e^{-B\abs{z}^2} \sum_{i=1}^\infty \frac{(\abs{z}\abs{y})^{\alpha^i-\rpar{\frac{n}{2}-1}}}{2^{\alpha^i}\Gamma(1+\alpha^i)}
\end{align}
where the last inequality follows from \eqref{eq:yaglom_bound_pV}. From here, using bounds for the moment of gaussian random variables, we arrive at the bound:
\begin{align}
t^{\beta/2}\partial_np^\C(t+1,\sqrt{t}y,\xi) &\leq C_6 \sum_{i=1}^\infty \frac{C_7^{\alpha^i+n/2} \abs{y}^{\alpha^i-\rpar{\frac{n}{2}-1}}}{2^{\alpha^i/2}\Gamma(1+\alpha^i)} \Gamma\rpar{\frac{1+\alpha^i}{2}+\frac{n}{4}},
\end{align}
which is finite.

The same steps as above show that it is possible to apply the \dct\ in \eqref{eq:yaglom_normal}. Equation \eqref{eq:yaglomN} then follows from Lemma \ref{le:yaglomC}.
\qed

\subsection{Proof of Theorem \ref{th:mainyaglom}}

\proof As before, our starting point is equation \eqref{eq:strongmarkov}. We will study the rate of decay of the integral involved in such equation by splitting in two terms, as before.
%
First, let us study
\begin{align}
t^{\beta/2} I_1(t) = t^{\beta/2} \int_0^{t/2} \int_{\bd_j} \frac{1}{2} \partial_n p^j(t-s,a_j+\sqrt{t}y,\xi) p(s,x,\xi) d\xi ds.
\end{align}
Using Lemma \eqref{le:yaglomN}, we see that we can apply the \dct\ to this integral to obtain
\begin{align}
\nonumber
\lim_{t\to\infty} t^{\beta/2} I_1(t) &= \ind_\M(j) \int_0^\infty\int_{\bd_j} \frac12 \frac{\partial_n w_j(\xi)e^{-\abs{y}^2/2} v_j(y)}{2^{\alpha}\Gamma(1+\alpha)} p(s,x,\xi) d\xi ds \\
\nonumber
&=  \ind_\M(j) \frac{e^{-\abs{y}^2/2} v_j(y)}{2^{\alpha}\Gamma(1+\alpha)}\int_{\bd_j} \frac12 \partial_n w_j(\xi) G(x,\xi) d\xi ds \\
\label{eq:yaglom_I1}
&= \ind_\M(j) \frac{e^{-\abs{y}^2/2} v_j(y)}{2^{\alpha}\Gamma(1+\alpha)} \rpar{u_j(x)-w_j(x)}.
\end{align}

Second, we look at
\begin{align}
t^{\beta/2} I_2(t) &= t^{\beta/2} \int_0^{t/2} \int_{\bd_j} \frac{1}{2}  \partial_n p^j(s,a_j+\sqrt{t}y,\xi) p(t-s,x,\xi) d\xi ds \\
\nonumber
&= t^{-\kappa/2} \int_0^{t/2} \int_{\bd_j} \frac{1}{2}  \partial_n p^j(s,a_j+\sqrt{t}y,\xi) t^{1+\alpha}p(t-s,x,\xi) d\xi ds
\end{align}
From Theorem \ref{th:mainkernel}, we have that $t^{1+\alpha}p(t-s,x,\xi)\leq C_x$ for all $s\in [0,t/2]$ as long as $t>3$. The constant $C_x$ depends only on $\abs{x}$. Then
\begin{align*}
t^\beta I_2(t)  &\leq C_x t^{-\kappa/2} \int_0^{t/2} \int_{\bd_j} \frac{1}{2}  \partial_n p^j(s,a_j+\sqrt{t}y,\xi) d\xi ds \\
&= C_x t^{-\kappa/2}  \PP_{\sqrt{t}y}(B_{T^j} \in \bd_j,T^j < t/2) \leq C_x t^{-\kappa/2},
\end{align*}
which converges to zero.

Putting together equation \eqref{eq:strongmarkov}, Lemma \ref{le:yaglomC}, and the last estimates, we obtain
\begin{align}
\lim_{t\to \infty} t^{\beta/2} p(t,x,a_j + \sqrt{t}y) &= \ind_{\M}(j) \frac{u_j(x)v_j(y)}{\zed} e^{-\abs{y}^2/2},
\end{align}
as desired.
\qed

\subsection{Distributional convergence of the renormalized process}
\newcommand{\mus}{n/2}

Theorem \ref{th:mainyaglom} suggests that, when conditioned on survival, most of the trajectories  of Brownian motion at time $t$ stay within order $\sqrt{t}$ from the origin. Thus, it is natural to study the convergence of the rescaled process $B_t/\sqrt{t}$ conditioned on survival. 

Let $A\sub V_j$ be a precompact, Borel set. Notice that $\beta-\kappa=n$. Then, by a simple change of variable
\begin{align*}
\PP_x\rpar{ (B_t-a_j)/\sqrt{t}\in A \vert T>t} &= \int_{\sqrt{t}A} \frac{ p(t,x,a_j+z) }{\PP_x(T>t)} dz \\
&= \int_A \frac{p(t,x,a_j+\sqrt{t}y)}{\PP_x(T>t)} t^{n/2}dy \\
&= \int_A \frac{t^{\beta/2} p(t,x,a_j+\sqrt{t}y)}{t^{\kappa/2}\PP_x(T>t)} dy,
\end{align*}
By Theorems \ref{th:mainexit} and \ref{th:mainyaglom}, the integrand on the right hand side converges to the function
\begin{align}
\label{eq:yaglom_limit}
p_x(j,y) &=  \frac{v_j(y)e^{-\abs{y}^2/2}}{\gamma_j\zed}  \cdot \frac{\gamma_j u_j(x)}{\sum_{k\in\M} \gamma_k u_k(x)} \ind_{V_j}(y).
\end{align}
Equation \eqref{eq:yaglom_limit} defines a  probability distribution function on $\M\times \cup_{j\in\M} V_j$ for a family of random variables $X^x=(X^x_1,X^x_2)$, with $x\in\O$, which is simple to interpret. Fix $x\in\O$ and let $X^x_1$ be a discrete random variable with distribution given by
\begin{align}
\label{eq:yaglom_discrete}
\PP(X^x_1 = j) &= \frac{\gamma_j u_j(x)}{\sum_{k\in\M} \gamma_k u_k(x)},\qquad j\in\M.
\end{align}
This is a sample of one of the maximal branches of the multicone. As $t\to\infty$ the multicone $\O_j$ is rescaled into the cone with vertex $V_j$. Correspondingly, we define $X^x_2$ as a continuous random variable on $ \cup_{j\in\M} V_j$ satisfying
\begin{align}
\label{yaglom_conditional}
\PP(X^x_2 \in dy \vert X^x_1 = j) &=   \frac{v_j(y)e^{-\abs{y}^2/2}}{\gamma_j\zed} \ind_{V_j}(y).
\end{align}

Our computation at the beginning of the section, and the uniform convergence on compact sets shows that, under $\PP_x$,  the renormalized process $B_t/\sqrt{t}$ conditioned on survival converges weakly to $X^x$.

\bibliographystyle{plain}
\bibliography{bibliography}   

\end{document}